\documentclass[hidelinks,12pt,a4paper,reqno]{amsart}
\usepackage[utf8]{inputenc}
\usepackage{enumitem}
\usepackage{amsmath}
\usepackage{amsfonts}
\usepackage{amssymb}
\usepackage[all]{xy}
\usepackage{xfrac}
\usepackage[yyyymmdd,hhmmss]{datetime}

\usepackage[english]{babel}
\usepackage{dsfont}
\usepackage{bbold}
\usepackage{soul}

\usepackage[backref=page]{hyperref}
\renewcommand*{\backref}[1]{}
\renewcommand*{\backrefalt}[4]{%
    \ifcase #1 (Not cited.)%
    \or        (Cited on page~#2.)%
    \else      (Cited on pages~#2.)%
    \fi}

\newtheorem{theorem}{Theorem}[section]
\newtheorem{corollary}[theorem]{Corollary}
\newtheorem{lemma}[theorem]{Lemma}
\newtheorem{definition}{Definition}[section]
\newtheorem{proposition}[theorem]{Proposition}

\newtheorem{example}{Example}[section]

\newcommand{\R}{\mathbb{R}}

\newcommand{\muu}{\text{\sfrac{1}{2}}}
\newcommand{\muuu}{\frac{\mathbb{1}}{\mathbb{2}}}

 \author{J. P. Fatelo and N. Martins-Ferreira}
 \address{School of Technology and Management, Centre for Rapid and Sustainable Product Development - CDRSP, Polytechnic Institute of Leiria, P-2411-901 Leiria, Portugal.}
 \email{martins.ferreira@ipleiria.pt}
\title[]{Mobility spaces and their geodesic paths}


\thanks{
This work is supported by the Fundação para a Ciência e a Tecnologia (FCT) and Centro2020 through the Project references: UID/Multi/04044/2019; PAMI - ROTEIRO/0328/2013 (Nº 022158); Next.parts (17963), and also by CDRSP and ESTG from the Polytechnic Institute of Leiria.
}

\begin{document}

\begin{abstract}

We introduce an algebraic system which can be used as a model for spaces with geodesic paths between any two of their points. This new algebraic structure is based on the notion of mobility algebra which has recently been introduced as a model for the unit interval of real numbers. Mobility algebras consist on a set $A$ together with three constants and a ternary operation. In the case of the closed unit interval $A=[0,1]$, the three constants are 0, 1 and 1/2 while the ternary operation is $p(x,y,z)=x-yx+yz$. A mobility space is a set $X$ together with a map $q\colon{X\times A\times X\to X}$ with the meaning that $q(x,t,y)$ indicates the position of a particle moving from point $x$ to point $y$ at the instant $t\in A$, along a geodesic path within the space $X$. A mobility space is thus defined with respect to a mobility algebra, in the same way as a module is defined over a ring.
We introduce the axioms for mobility spaces, investigate the main properties and give examples.
    We also establish the connection between the algebraic context and the one of spaces with geodesic paths. The connection with affine spaces is briefly mentioned.

\end{abstract}

\maketitle

\today; \currenttime

\section{Introduction}

The purpose of this work is to introduce an algebraic system which can be used to model spaces with geodesics. The main idea stems from the interplay between algebra and geometry. In Euclidean geometry the notion of affine space is well suited for this purpose. Indeed, in an affine space we have scalar multiplication, addition and subtraction and so it is possible to parametrize, for any instant $t\in [0,1]$, a straight line between points $x$ and $y$ with the formula $(1-t)x+ty$. Such a line is clearly a geodesic path from $x$ to $y$. 
In general terms, we may use an operation $q=q(x,t,y)$ to indicate the position, at an instant $t$, of a particle moving in space from a point $x$ to a point $y$. If the particle is moving along a geodesic path then this operation must certainly verify some conditions. The aim of this project is to present an algebraic structure, $(X,q)$, with axioms that are verified by any operation $q$ representing a geodesic path in  a space between any two of its points. 

First results concerning this investigation were presented in \cite{ccm_magmas} where a binary operation, obtained by fixing $t$ to a value that positions the particle at { half way} from $x$ to $y$, is studied. The whole movement of a particle on a geodesic path is captured when the variable $t$ is allowed to range over a set of values, of which the unit interval is the most natural choice. The investigation of those structures and properties relevant to our study led us to the discovery of a new algebraic structure that was called mobi algebra \cite{mobi}.
A mobility algebra (or mobi algebra), besides being a suitable algebraic model to the unit interval, offers an interesting comparison with rings. A slogan may be used to illustrate that comparison:  \emph{a mobi algebra is to the unit interval in the same way as a ring is to the set of reals}. A mobi algebra is an algebraic system $(A,p,0,\muu,1)$ consisting of a set together with a ternary operation $p$ and three constants. Moreover, every ring in which $2$ is invertible has a mobi algebra structure, while a mobi algebra in which the element $\muu$ is invertible (as defined in~\cite{mobi}) has  a ring structure.

Following the analogy with rings, and extending it to modules over a ring, we have considered a new structure, called  mobility space (mobi space for short). If $A$ is a mobi (or mobility) algebra then a mobi space, say $(X,q)$, is defined over a mobi algebra in the sense that $q=q(x,t,y)$ operates on $x,y\in X$ and $t\in A$. The axioms defining a mobi space are similar to the ones defining a mobi algebra and examples show that these axioms are indeed appropriate to capture the features of spaces with geodesic paths between any two of their points.

This paper is organized as follows. The notion of  mobi algebra \cite{mobi} is recalled together with its derived operations (Section \ref{sec_mobi_alg}). Two important new sorts of examples are added to the  list presented in \cite{mobi}. A first sort of examples is constructed (Proposition~\ref{mobiEnd}) from a midpoint algebra by considering a suitable subset of its endomorphisms (Definition \ref{End_e}). Another sort of examples  is generated by considering a unitary ring in which the element 2 is invertible. As  a particular case, the set of endomorphisms of an abelian group is canonically equipped with a unitary ring structure and hence, if the map $x\mapsto 2x$ is invertible, then it gives rise to a mobi algebra structure on its set of endomorphisms (Proposition \ref{prop_endX}). These examples will be used  to characterize affine mobi spaces which include, among other things, all vector spaces over a field and will be considered in a sequel of this work. Nonetheless, our last section is dedicated to a thoroughly comparison between affine mobi spaces and modules over a ring (Theorem \ref{module2mobi} and Theorem \ref{mobi2module}). In addition, some details are given in the form of examples in subsection \ref{sec_example-midpoint}.

 Section \ref{sec_mobi_space} is totally devoted to the definition of mobi space and a short list of its properties that will be needed in this paper. Further studies are postponed to a future work, namely the transformations between such structures, and the fact that the category of mobi spaces is a weakly Mal'tsev category \cite{ccm_magmas,NMF.08,NMF.15a}. Or its comparison with the work of A. Kock on affine connections with neighbouring relations \cite{Kock,Kock3,Kock2} and its relation to the work of Buseman  \cite{busemann-1943,busemann-1955} on spaces with unique geodesic paths. 
 
 Section 4 is dedicated to an exhaustive list of worked examples and counter examples that have been used for testing the strength of our axioms. These include simple examples (Example \ref{ex1to4}), more complex examples (Example \ref{ex5to9}), the particular case of affine mobi spaces (which are comparable to euclidean geometry) and some examples with a physical interpretation (Examples \ref{ex11} and \ref{ex:fxdot}). In particular, we observe that even for other type of physical equations, rather than those for geodesics, there is a procedure which can be followed to include physical phenomenon into a mobility space. The procedure is based on the simple trick of considering time as a geometrical dimension (see Example \ref{ex:fxdot} and Corollary \ref{corollary-xdot}). 

 In Section \ref{geo} we study the main examples which arise as the solution to the equation for geodesics and extend that result to cover a wide range of cases in which geodesics are not uniquely determined. This is done via the process of identification spaces.
In the end we have added a section with a comparison between affine mobi spaces and modules over a unitary ring in which the element $2$ is invertible.

\section{Mobi algebra}\label{sec_mobi_alg}

In this section we briefly recall the notion of mobi algebra which was introduced in \cite{mobi} and some of its basic properties. A mobi algebra consists of a set equipped with three constants and one ternary operation. The unit interval $[0,1]$, with the three constants $0$, $\frac{1}{2}$, $1$ and  the formula $p(a,b,c)=a+bc-ba$ has always been the example of  a mobi algebra which have driven our intuition. Nevertheless, several other examples  of very different nature were presented in  \cite{mobi}. At the end of this section we add two new important classes to the list of examples. As we will see in the next section (while introducing the notion of mobi space) a mobility algebra plays the role of scalars (for a mobility space) in the same way as a ring (or a field) models the scalars for a module (or a vector space) over the base ring (or field). 

 In order to have an intuitive interpretation of its axioms we may consider a mobi algebra as a mobi space over itself and use the geometric intuition provided in section \ref{sec_mobi_space}.
Namely, that the operation $p(x,t,y)$ is the position of a particle moving from a point $x$ to a point $y$ at an instant $t$ while following a geodesic path.

\begin{definition}\cite{mobi}\label{mobi_algebra}
A mobi algebra is a system $(A,p,0,\muu,1)$, in which $A$ is a set, $p$ is a ternary operation and $0$, $\muu$ and $1$ are elements of A, that satisfies the following axioms:
\begin{enumerate}[label={\bf (A\arabic*)}]
\item\label{alg_mu} $p(1,\muu,0)=\muu$
\item\label{alg_01} $p(0,a,1)=a$
\item\label{alg_idem} $p(a,b,a)=a$
\item\label{alg_0} $p(a,0,b)=a$
\item\label{alg_1} $p(a,1,b)=b$
\item\label{alg_cancel} $p(a,\muu,b_1)=p(a,\muu,b_2)\implies b_1=b_2$
\item\label{alg_homo} $p(a,p(c_1,c_2,c_3),b)=p(p(a,c_1,b),c_2,p(a,c_3,b))$
\item\label{alg_medial}
$p(p(a_1,c,b_1),\muu,p(a_2,c,b_2))=p(p(a_1,\muu,a_2),c,p(b_1,\muu,b_2))$.
\end{enumerate}
\end{definition}

Some properties of mobi algebras can be suitably expressed in terms of a unary operation~"$\overline{()}$" and binary operations~"$\cdot$", "$\circ$" and "$\oplus$" defined as follows (see \cite{mobi} for more details).
\begin{definition}\label{binary_definition} Let $(A,p,0,\muu,1)$ be a mobi algebra. We define:
\begin{eqnarray}
\label{def_complementar}\overline{a}&=&p(1,a,0)\\
\label{def_product}a\cdot b&=&p(0,a,b)\\
\label{def_star}a\oplus b&=&p(a,\muu,b)\\
\label{def_oplus}a\circ b&=&p(a,b,1).
\end{eqnarray}
\end{definition}

We recall a list of some properties of a mobi algebra. If $(A,p,0,\muu,1)$ is a mobi algebra, then: 
\begin{eqnarray}
\label{Bmuu} \overline{\muu}&=&\muu\\
\label{B130} a\cdot\muu=\muu\cdot a&=&0\oplus a\\
\label{cancell-muu} \muu\cdot a=\muu\cdot a'&\Rightarrow& a=a'\\
\label{B7} p(\overline{a},\muu,a)&=&\muu\\
\label{Boverlinea} \overline{a}=a &\Rightarrow & a=\muu\\
\label{complementary-p} \overline{p(a,b,c)}&=&p(\overline{a},b,\overline{c})\\
\label{commut} p(c,b,a)&=&p(a,\overline{b},c)\\
\label{circ_cdot} \overline{a\circ b}&=&\overline{b}\cdot\overline{a}\\
\label{important}\muu\cdot p(a,b,c)&=&(\overline{b}\cdot a)\oplus (b\cdot c).
\end{eqnarray}

A long list of examples is provided in \cite{mobi}. The following example shows that any non-empty midpoint algebra $(X,\oplus)$, with a chosen element $e\in X$ in it, gives rise  to a mobi algebra structure on a suitable subset of endomaps of $X$. This example will be crucial for the characterization of affine mobi spaces (see Section \ref{sec_example-midpoint}).

Let us recall that a midpoint algebra \cite{Escardo} consists of a set $X$ and a binary operation 
$\oplus$ satisfying the following axioms:
\[\begin{array}{rl}
 \text{(idempotency)} & x\oplus x= x\\
 \text{(commutativity)} & x\oplus y= y \oplus x\\
 \text{(cancellation)} & x\oplus y= x' \oplus y\Rightarrow x=x'\\
\text{(mediality)}&(x\oplus y)\oplus (z\oplus w)=( x \oplus z)\oplus (y\oplus w) .
\end{array}\]
 It is remarkable that a certain subset of endomaps of $X$ can always be given the structure of a mobi algebra.

\begin{definition}\label{End_e}
Let $(X,\oplus)$ be a non-empty midpoint algebra and \mbox{$e\in X$} a fixed chosen element in $X$. We denote by $End_{e}^{\oplus}(X)$ the set of endomaps of $X$ consisting of those maps $g\colon{X\to X}$ satisfying the following conditions:
\begin{enumerate}
\item[(i)] $g(x\oplus y)=g(x)\oplus g(y)$,  for all $x,y\in X$,
\item[(ii)] $g(e)=e$
\item[(iii)] for all $x\in X$, there exists $\bar{g}(x)\in X$ such that $\bar{g}(x)\oplus g(x)=e\oplus x$
\item[(iv)] for all $x,y\in X$, there exists $\tilde{g}(x,y)\in X$ such that $\bar{g}(x)\oplus g(y)=e\oplus\tilde{g}(x,y)$
\end{enumerate}
\end{definition}

Let us observe that when $g\in End_e^{\oplus}(X)$ then $\bar{g}$, defined as in $(iii)$, is also in $End_e^{\oplus}(X)$. Moreover, we have $\bar{\bar{g}}=g$ and $\tilde{\bar{g}}(x,y)=\tilde{g}(y,x)$.
An example is provided in subsection \ref{sec_example-midpoint}, where affine mobi spaces are considered.

\begin{proposition}\label{mobiEnd}
Let $(X,\oplus)$ be a non-empty midpoint algebra and fix an element $e\in X$. The system  $(End_e^{\oplus}(X), \mathbb{p}, \mathbb{0}, \muuu,\mathbb{1})$ , with $\mathbb{p}(x)=\mathbb{p}_{f,g,h}(x)=\tilde{g}(f(x),h(x))$, $\mathbb{0}(x)=e$, $\muuu(x)=e\oplus x$ and $\mathbb{1}(x)=x$  is a mobi algebra.
\end{proposition}
\begin{proof}
First we show that the operation $\mathbb{p}$ is well defined. To do that we observe that if $g\in  End_{e}^{\oplus}(X)$ then $\tilde{g}$, defined as in $(iv)$, is such that
\begin{eqnarray}
\tilde{g}(x_1\oplus x_2,y_1\oplus y_2)&=&\tilde{g}(x_1,y_1)\oplus \tilde{g}(x_2,y_2)\\
\tilde{g}(x,x)&=&x
\end{eqnarray}
Now, given $f,g,h\in End_{e}^{\oplus}$ we have to show that $\mathbb{p}_{f,g,h}$, defined for every $x\in X$ as $\tilde{g}(f(x),h(x))$, belongs to $End_{e}^{\oplus}$. The first two conditions $(i)$ and $(ii)$ follow from the properties of $\tilde{g}$. Condition $(iii)$ is satisfied with $\bar{\mathbb{p}}_{f,g,h}(x)=\tilde{g}(\bar{f}(x),\bar{h}(x))$, while condition $(iv)$ is satisfied with $\tilde{\mathbb{p}}_{f,g,h}(x,y)=\tilde{g}(\tilde{f}(x,y),\tilde{h}(x,y))$.
The fact that the axioms of a mobi algebra (Definition \ref{mobi_algebra}) are verified by $\mathbb{p}_{f,g,h}(x)$ is then a consequence of Theorem 6.2 from \cite{mobi}. Indeed, the system $(End_e^{\oplus}(X), \overline{()},\oplus,\cdot,\mathbb{1})$ with $\overline{(g)}(x)=\overline{g}(x)$, $(f\oplus g)(x)=f(x)\oplus g(x)$, $(f\cdot g)(x)=f(g(x))$ and $\mathbb{1}(x)=x$ is a IMM* algebra in the sense of Definition 6.1 from \cite{mobi}. Moreover, the equation $\overline{\mathbb{1}}\oplus \chi=(\overline{g}\cdot f)\oplus (g\cdot h)$ has a solution in $End_e^\oplus(X)$, namely $\chi(x)=\tilde{g}(f(x),h(x))$ for every $x\in X$.
\end{proof}

As it was shown in \cite{mobi}, every unitary ring in which the element $2$ is invertible gives rise to a mobi algebra structure. A particular case, which will be used in section \ref{sec_example-midpoint}, is the ring of endomorphisms of an abelian group. Let $(X,+,0)$ be an abelian group. It is clear that the map $x\mapsto x+x$ is invertible if and only if for every $x\in X$, there exists $y\in X$ with $y+y=x$. For obvious reasons the element $y$ such that $y+y=x$ will be denoted $\frac{x}{2}$. As usual, we denote the set of group endomorphisms of $X$ by $End(X)$.

\begin{proposition}\label{prop_endX}
Let $(X,+,0)$ be an abelian group in which the endomorphism sending  $x$ to $x+x$ is invertible. The system $$(End(X), \mathbb{p}, \mathbb{0}, \muuu,\mathbb{1}),$$ with $\mathbb{p}(x)=\mathbb{p}_{f,g,h}(x)=f(x)-gf(x)+gh(x)$, $\mathbb{0}(x)=0$, $\muuu(x)=\frac{x}{2}$ and $\mathbb{1}(x)=x$  is a mobi algebra.
\end{proposition}
\begin{proof}
The result follows directly from  Theorem 7.2 in \cite{mobi} since $End(X)$ is a unitary ring with one half. 
\end{proof}

In spite of these somehow more abstract examples, the reader is kindly reminded that it is useful to keep in mind that the most natural and intuitive mobi algebra is the closed unit interval with $A=[0,1]$, the three constants $0$, $\frac{1}{2}$, $1$ and the operation \mbox{$p(a,b,c)=(1-b)a+bc$}. The mobi algebra, in this case, is interpreted as the set of scalars over which a mobi space is constructed.

\section{Mobi space}\label{sec_mobi_space}

In this section we give the definition of a mobi space over a mobi algebra. Its main purpose is to serve as a model for spaces with a geodesic path connecting any two points. It is similar to a module over a ring in the sense that it has an associated mobi algebra which behaves as the set of scalars. In Section \ref{sec_example-midpoint} (see also Section \ref{sec:modules} at the end of the paper) we show that the particular case of affine mobi space is indeed the same as a module over a ring when the mobi algebra is a ring.

\begin{definition}\label{mobi_space}
Let $(A,p,0,\muu,1)$ be a mobi algebra. An $A$-mobi space $(X,q)$, consists of a set $X$ and a map $q\colon{X\times A\times X\to X}$ such that:
\begin{enumerate}[label={\bf (X\arabic*)}]
\item\label{space_0} $q(x,0,y)=x$
\item\label{space_1} $q(x,1,y)=y$
\item\label{space_idem} $q(x,a,x)=x$
\item\label{space_cancel} $q(x,\muu,y_1)=q(x,\muu,y_2)\implies y_1=y_2$
\item\label{space_homo} $q(q(x,a,y),b,q(x,c,y))=q(x,p(a,b,c),y)$
\end{enumerate}
\end{definition}

The axioms \ref{space_0} to \ref{space_homo} are the natural generalizations of axioms \ref{alg_idem} to \ref{alg_homo} of a mobi algebra. The natural generalization of \ref{alg_medial} would be
\begin{eqnarray}\label{affine}
q(q(x_1,a,y_1),\muu,q(x_2,a,y_2))=q(q(x_1,\muu,x_2),a,q(y_1,\muu,y_2)).
\end{eqnarray}
This condition, however, is too restrictive and is not in general verified by geodesic paths. That is the reason why we do not include it.
When (\ref{affine}) is satisfied for all $x_1,x_2,y_1,y_2\in X$ and $a\in A$, then we say that the $A$-mobi space $(X,q)$ is affine and speak of an $A$-mobi affine space (see Section \ref{sec_example-midpoint}).

Here are some immediate consequences of the axioms for a mobi space.

\begin{proposition}\label{properties_space} Let $(A,p,0,\muu,1)$ be a mobi algebra and $(X,q)$ an A-mobi space. It follows that:
\begin{enumerate}[label={\bf (Y\arabic*)}]
\item\label{Y1} $q(y,a,x)=q(x,\overline{a},y)$
\item\label{Y2} $q(y,\muu,x)=q(x,\muu,y)$
\item\label{Y3} $q(x,a,q(x,b,y))=q(x,a\cdot b,y)$
\item\label{Y4} $q(q(x,a,y),b,y)=q(x,a\circ b,y)$
\item\label{Y7} $q(q(x,a,y),\muu,q(x,b,y))=q(x,a\oplus b,y)$
\item\label{Y6} $q(x,\muu,q(x,a,y))=q(x,a,q(x,\muu,y))$
\item\label{Y5} $q(q(x,a,y),\muu,q(y,a,x))=q(x,\muu,y)$
\item\label{Y8} $q(q(q(x,a,y),b,x),\muu,q(x,b,q(x,c,y)))\\
=q(x,\muu,q(x,p(a,b,c),y))$
\item\label{Y9} $q(x,a,y)=q(y,a,x)\Rightarrow q(x,a,y)=q(x,\muu,y)$
\item\label{Y10} $q(x,a,y)=q(x,b,y)\Rightarrow q(x,p(a,t,b),y)=q(x,a,y)$, \text{for all $t$}.
\end{enumerate}
\end{proposition}

\begin{proof}
The following proof of \ref{Y1}, bearing in mind (\ref{def_complementar}), uses \ref{space_homo}, \ref{space_1} and \ref{space_0}:
\begin{eqnarray*}
q(x,\overline{a},y)&=&q(x,p(1,a,0),y)\\
                   &=&q(q(x,1,y),a,q(x,0,y))\\
                   &=&q(y,a,x).
\end{eqnarray*}
\ref{Y2} follows directly from \ref{alg_mu} and \ref{Y1}.
Begining with (\ref{def_product}), \ref{Y3} is a consequence of \ref{space_homo} and \ref{space_0}:
\begin{eqnarray*}
q(x,a\cdot b,y)&=&q(x,p(0,a,b),y)\\
          &=&q(q(x,0,y),a,q(x,b,y))\\
          &=&q(x,a,q(x,b,y)).
\end{eqnarray*}
Considering (\ref{circ_cdot}), property \ref{Y4} follows from \ref{Y1} and \ref{Y3}.
\begin{eqnarray*}
q(q(x,a,y),b,y)&=&q(y,\overline{b},q(y,\overline{a},x))\\
               &=&q(y,\overline{b}\cdot\overline{a},x)\\
               &=&q(y,\overline{a\circ b},x)\\
               &=&q(x,a\circ b,y).
\end{eqnarray*}
Considering $(\ref{def_oplus})$, \ref{Y7} is just a particular case of \ref{space_homo}.
To prove \ref{Y6}, we use \ref{space_homo}, (\ref{B130}) and \ref{space_0}.
\begin{eqnarray*}
q(x,\muu,q(x,a,y))&=&q(q(x,0,y),\muu,q(x,a,y))\\
                 &=&q(x,p(0,\muu,a),y)\\
                 &=&q(x,p(0,a,\muu),y)\\
                 &=&q(q(x,0,y),a,q(x,\muu,y))\\
                 &=&q(x,a,q(x,\muu,y)).
\end{eqnarray*}
The following proof of \ref{Y5} is based on \ref{Y1}, \ref{space_homo} and (\ref{B7});
\begin{eqnarray*}
q(q(x,a,y),\muu,q(y,a,x))&=&q(q(x,a,y),\muu,q(x,\overline{a},y))\\
                        &=&q(x,p(a,\muu,\overline{a}),y)\\
                        &=&q(x,\muu,y).
\end{eqnarray*}
To prove \ref{Y8}, we start with the important property (\ref{important}) of the underlying mobi algebra and then get:
\begin{eqnarray*}
& &q(x,\muu \cdot p(a,b,c),y)=q(x,\overline{b}\cdot a\oplus b\cdot c,y)\\
&\Rightarrow& q(x,p(0,\muu,p(a,b,c)),y)=q(x,p(\overline{b}\cdot a, \muu,b\cdot c),y)\\
&\Rightarrow& q(x,\muu,q(x,p(a,b,c),y))=q(q(x,\overline{b}\cdot a,y),\muu,q(x,b\cdot c,y))\\
&\Rightarrow& q(x,\muu,q(x,p(a,b,c),y))\\
& &=q(q(x,\overline{b},q(x,a,y)),\muu,q(x,b,q(x,c,y))).
\end{eqnarray*}
It is easy to see that $\ref{Y9}$ is a direct consequence of $\ref{Y5}$, while $\ref{Y10}$ is a consequence of \ref{space_idem} and \ref{space_homo}. 
\end{proof}

The property $\ref{Y9}$ shows that if someone is travelling from $x$ to $y$ and someone else is travelling from $y$ to $x$ (along a unique geodesic path parametrized by $q$ that connects $x$ and $y$) then they meet at $q(x,\muu,y)$.
The element $\muu$ in a mobility algebra is understood as the point at which $p(a,\muu,b)$ equals  $p(b,\muu,a)$. When interpreted in terms of a parameter on the parametrized operation $q(x,\muu,y)$ on a mobility space it has the meaning that, although it  may not be one half of the path between $x$ and $y$ in the expected sense (see Examples 2 and 3 in \cite{mobi}), it is nevertheless the position at half way in the sense of a metric in a metric space. Property $\ref{Y10}$ shows that if a particle is at the same place at two distinct moments $a$ and $b$, then it stays there at any instant in-between.

A related property, which is perhaps worthwhile studying, is the following one: 
\begin{equation}\label{eq: conjecture}
\text{if there exist}\ a\neq b \  \text{with}\ q(x,a,y)=q(x,b,y)\ \text{then}\ x=y.
\end{equation}
This behaviour is desirable to model geodesics  as shortest paths but it cannot be deduced from the axioms of a mobi space. Indeed,
in Example \ref{ex_lozenge}, with $h=-1$, we have $q(x,(t,s),y)=(1-t+s)x+(t-s)y$ and hence, for every appropriate $u$, $x$ and $y$, $$q(x,(t+u,s+u),y)= q(x,(t,s),y).$$ For example, $q(x,(1,0),y)=q(x,(1/2,-1/2),y)=y$. Nevertheless, this property is satisfied when the mobi algebra is the unit interval and the mobi space is a space with unique geodesic paths. See the remark following Theorem \ref{prop unique beta}.

\section{Examples}\label{sec_examples}

In this section we give a list of examples and families of examples of mobi spaces over an appropriate mobi algebra. 

\subsection{First examples}
In the list of examples below, the underlying mobi algebra structure $(A, p, 0, \muu, 1)$ is the closed unit interval, i.e. $A=[0,1]$, the three constants are $0, \frac{1}{2}, 1$ and 
$$p(a,b,c)=(1-b)a+b c,\ \textrm{for all}\ a,b,c\in A.$$ In each case, we present a set $X$ and a ternary operation $q(x,t,y)\in X$, for all $x,y\in X$, and $t \in A$, verifying the axioms of Definition \ref{mobi_space}. 
\begin{example}\label{ex1to4}
Mobi spaces $(X,q)$ over the unit interval:
\begin{enumerate}
\item \label{ex1} Euclidean spaces provide examples of the form
\[X=\R^n\quad (n\in\mathbb{N})\]
and
\[q(x,t,y)=(1-t)x+t\,y.\]

\item \label{ex2} The positive real numbers with the metric inspired by the usual euclidean distance after applying a transformation $x\mapsto \exp(x)$ provides an example of the form
\[X=\R^+\]
with
\[q(x,t,y)=x^{1-t} y^t.\]

\item \label{ex3} Another example, still inspired by the euclidean metric but with a different transformation, is
\[X=\R^+\]
\[q(x,t,y)=\dfrac{x y}{t x+(1-t) y}.\]

\item \label{ex4} The previous examples are all particular instances of a family of examples which can be constructed as follows. Start with  a function $F$, invertible on a subset of the real numbers $X\subseteq \R$, and put
\begin{eqnarray*}
q(x,t,y)=F^{-1}((1-t)F(x)+t\,F(y)).
\end{eqnarray*}
For instance, Example \ref{ex1to4}(\ref{ex2}) corresponds to $F(x)=\log x$ while Example \ref{ex1to4}(\ref{ex3}) to $F(x)=\frac{1}{x}$. Clearly, in Example \ref{ex1to4}(\ref{ex1}), $F(x)=x$.
\end{enumerate}
\end{example}

\subsection{Other type of examples}

The canonical formula for euclidean spaces, considered in Example \ref{ex1to4}(\ref{ex1}) of the previous subsection, can be adapted to provide examples of mobi spaces in subsets of $\R^n$.
\begin{example}\label{ex5to9}
Further examples of mobi spaces $(X,q)$ over the unit interval.
\begin{enumerate}
\item \label{ex5} The two subsets of the plane 
\[X=]-\pi,\pi]\times\R\ \textrm{or}\ X=[0,2\pi[\times\R\]
with the formula
\[q((\theta_1,z_1),t,(\theta_2,z_2))=((1-t)\theta_1+t\,\theta_2,(1-t) z_1+t z_2)\]
can be viewed as two different choices of geodesic paths on the cylinder
\[\{(x,y,z)\in\R^3\mid x=\cos\theta, y=\sin\theta, (\theta,z)\in X\},\]
from the point $(\cos\theta_1,\sin\theta_1,z_1)$ to the point $(\cos\theta_2,\sin\theta_2,z_2)$. For instance, if considering the point in the cylinder $$(\cos(-\frac{\pi}{4}),\sin(-\frac{\pi}{4}),0)=(\cos(\frac{7\pi}{4}),\sin(\frac{7\pi}{4}),0)$$ and the point $$(\cos(\frac{\pi}{4}),\sin(\frac{\pi}{4}),0),$$ the two different parametrizations give two different paths between them. Indeed, one path goes through  $(1,0,0)$ while the other goes thorough $(-1,0,0)$. A third choice for a parametrization, which corresponds to the shortest paths for any two points on the cylinder, will be given  in subsection \ref{subsec: identification spaces}.

\item \label{ex6} Other examples are also possible, such as the set
\[X=\mathbb{R}^+\times\mathbb{R}\]
with the formula
\[q((x_1,x_2),t,(y_1,y_2))=(x_1+(y_1-x_1)t,x_2+(y_2-x_2)\frac{y_1\,t}{(1-t)x_1+t\,y_1}).\]

\item \label{ex7} Or the set
\[X=\mathbb{R}^+\times\mathbb{R}\]
with the formula
\[q((x_1,x_2),t,(y_1,y_2))=(x_1+(y_1-x_1)t,x_2+\frac{y_2-x_2}{x_1+y_1}(2\,x_1 t+(y_1-x_1) t^2)).\]

\item \label{ex8} We can also consider the set
\[X=\mathbb{R}^2\]
and the formula
\[q((x_1,x_2),t,(y_1,y_2))=\]
\[(x_1+(y_1-x_1)t,x_2+(y_2-x_2)\frac{3x_1^2\,t+3x_1(y_1-x_1)t^2+(y_1-x_1)^2\,t^3}{x_1^2+x_1y_1+y_1^2}),\]
if $(x_1,y_1)\neq (0,0)$, and
\[q((0,x_2),t,(0,y_2))=\left(0,x_2+t\,(y_2-x_2)\right).\]

\item \label{ex9} The last three examples are just particular cases of the following family of mobi spaces, where $f$ is an injective real function of one variable and $X\subseteq\mathbb{R}^2$ is  any set for which the formula 
\[q((x_1,x_2),t,(y_1,y_2))=\]
\[ \left(x_1+(y_1-x_1)t,x_2+(y_2-x_2)\frac{f(x_1+(y_1-x_1) t)-f(x_1)}{f(y_1)-f(x_1)}\right),\]
if $x_1\neq y_1$, and
\[q((x,x_2),t,(x,y_2))=\left(x,x_2+t\,(y_2-x_2)\right),\]
defines a map $q\colon{X\times[0,1]\times X\to X}$.   \end{enumerate}
\end{example}
Examples \ref{ex5to9}(\ref{ex6}), (\ref{ex7}) and (\ref{ex8}), are obtained as particular cases of Example \ref{ex5to9}(\ref{ex9}), respectively, with $f(x)=\frac{1}{x}$, $f(x)=x^2$ and $f(x)=x^3$.


So far we have considered examples of mobi spaces over the unit interval. Here is an example with a different mobi algebra.
\begin{example} \label{ex12}\label{ex_lozenge}
 For the mobi algebra $(A,p,0,\muu,1)$ let us use
 $$A=\left\{(t_1,t_2)\in\mathbb{R}^2\colon \vert t_2\vert \leq t_1 \leq 1-\vert t_2\vert\right\}$$
$$\muu=\left(\frac{1}{2},0\right)\,;\,1=(1,0)\,;\,0=(0,0)$$
\begin{eqnarray*}
p(a,b,c)=(a_1-b_1 a_1-b_2 a_2+b_1 c_1+b_2 c_2,\\
          a_2-b_1 a_2-b_2 a_1+b_1 c_2+b_2 c_1).
\end{eqnarray*}
And for the mobi space $(X,q)$, let us use
$X=[0,1]$ and $$q(x,(t,s),y)= (1-t-h\,s) x+(t+h\,s)y,$$ with $h=\pm 1$.

\end{example}

Let us now turn to a special case of mobi spaces, namely affine mobi spaces.

\subsection{Affine mobi spaces}\label{sec_example-midpoint}

Let $(A,p,0,\muu,1)$ be a mobi algebra. An affine mobi space (over $A$) is a mobi space $(X,q)$ for which the condition 
\begin{equation}\label{eq: affine mobi space}
q(q(x_1,a,y_1),\muu,q(x_2,a,y_2))=q(q(x_1,\muu,x_2),a,q(y_1,\muu,y_2))
\end{equation}
 is satisfied for every $x_1,x_2,y_1,y_2\in X$ and $a\in A$.

We observe that the (even) stronger condition 
\begin{equation}\label{eq: commutative affine mobi space}
q(q(x_1,a,y_1),b,q(x_2,a,y_2))=q(q(x_1,b,x_2),a,q(y_1,b,y_2))
\end{equation}
 for every $x_1,x_2,y_1,y_2\in X$ and $a,b\in A$, is worthwhile studying due to its connection with Proposition 6.4 in \cite{mobi}.

If $(X,q)$ is an affine mobi space then we obtain a midpoint algebra $(X,\oplus)$ by defining $x\oplus y=q(x,\muu,y)$ (see \cite{ccm_magmas} for the particulars of the binary operation $\oplus$). In a future work we will investigate the converse. That is,  given a midpoint algebra $(X,\oplus)$, we will study under which conditions it is obtained from an affine mobi space. Another interesting topic is to study the collection of all mobi space structures which give rise to the same midpoint algebra, developing the concept of homology for mobi affine spaces.

For the moment, let us mention that when we choose an origin $e\in X$ in a given (non-empty) midpoint algebra and assuming that there is an abelian group structure associated to it (in the sense of \cite{ccm_magmas}), then, the set of group endomorphisms on $X$ has a ring structure and, by the use of Proposition \ref{prop_endX}, we conclude that it gives rise to a mobi algebra. Recall that a midpoint algebra $(X,\oplus)$ with an associated abelian group is a midpoint algebra with the property that for every $x,y\in X$ there is an element $(x+y)\in X$ such that $e\oplus(x+y)=x\oplus y$, then it follows that  $(X,+,e)$ is an abelian group with the property that the map $x\mapsto x+x$ is invertible, with inverse $x\mapsto e\oplus x$.

 This simple observation gives rise to an important characterization, to be further developed into a future work, which can be stated as follows:

\begin{quotation}

The collection of mobi algebra homomorphisms from $A$ to $End(X)$ is in a one-to-one correspondence with the collection of all affine mobi spaces $(X,q)$ such that $\oplus$ is determined by $q$.

\end{quotation}

 It is remarkable how similar the notion of affine mobi space is from the notion of an affine vector space. In particular, when $A$ has an inverse to $\muu$, say $2\in A$, then the operation $+$ always exists and is explicitly given by the formula $x+y=q(e,2,q(x,\muu,y))$. In this case, we have precisely the notion of an $A$-module (see the last Section).

Surprisingly, even when $(X,\oplus)$  does not allow an abelian group structure, a similar characterization is still possible, at the expense of replacing the ring  $End(X)$ with a more sophisticated structure, namely  $End_e^{\oplus}(X)$  which is specially designed (Definition \ref{End_e}) to keep the analogy with the previous result while extending it into a more general setting.

 \begin{quotation}
 The collection of mobi algebra homomorphisms from $A$ to $End_e^{\oplus}(X)$ is in a one-to-one correspondence with the collection of all affine mobi spaces $(X,q)$ such that $\oplus$ is determined by $q$.
 \end{quotation}

This and other aspects of affine mobi spaces will be developed thoroughly in the continuation of this study.
However, in order to give a glimpse of the kind of results that are expected, we present an example of a mobi space constructed from a midpoint algebra.

\begin{example} \label{ex10} 
 Consider the midpoint algebra $(X,\oplus)$, with $X=[0,1]$ and $x\oplus y=\dfrac{x+y}{2}$ for every $x,y\in X$. The fixed element in $X$ is chosen to be $e=0$. The mobi algebra is $(A,p,(0,0),(\frac{1}{2},0),(1,0))$ with
$$A=\left\{(x,y)\in\mathbb{R}^2\mid \vert y\vert \leq x \leq 1-\vert y\vert\right\}$$ 
and
\begin{eqnarray*}
p(a,b,c)=(a_1-b_1 a_1-b_2 a_2+b_1 c_1+b_2 c_2,\\
          a_2-b_1 a_2-b_2 a_1+b_1 c_2+b_2 c_1).
\end{eqnarray*}

Let us now use Proposition \ref{mobiEnd}
 to generate an example of a mobi space from a homomorphism of mobi algebras.
It is easily checked that, for any $h\in[-1,1]$, the map
$$\varphi_{h,(a_1,a_2)}\colon{A\to End_e^\oplus(X)},$$ given by 
$$\varphi_{h,(a_1,a_2)}(x)=(a_1+ h a_2) x$$ is well defined in $End_e^\oplus(X)$. Indeed, Definition \ref{End_e} is verified with 
$$\overline{\varphi}_{h,(a_1,a_2)}(x)=(1-a_1- h a_2) x$$
and
$$\tilde{\varphi}_{h,(a_1,a_2)}(x,y)=(1-a_1-h a_2) x+(a_1+h a_2) y.$$ Note that $\overline{\varphi}_{h,(a_1,a_2)}=\varphi_{h,\overline{(a1,a2)}}$ because $\overline{(a_1,a_2)}=(1-a_1,-a_2)$ follows from (\ref{def_complementar}). However, $\varphi_{h,(a_1,a_2)}$ is an homomorphism of mobi algebras if and only if $h=\pm 1$. In these cases, the mobi operations on $X$, defined by
$$0\oplus q(x,(t,s),y)=\overline{\varphi}_{\pm 1,(t,s)}(x)\oplus\varphi_{\pm 1,(t,s)}(y),$$
are given by:
$$q(x,(t,s),y)=(1-t\mp s) x+(t\pm s) y.$$
\end{example}
The resulting mobi space $(X,q)$ is the same as the one in Example~\ref{ex12}.

We will now see some examples obtained from physics.

\subsection{Examples with physical interpretation}

This section ends with Example \ref{ex11}, obtained from the motion of a projectile, some comments on counter-examples, as well as a general example from mechanics.

\begin{example}\label{ex11}\label{projectiles} For any $k\in\R$, we may form a mobi space $(X,q)$ over the unit interval by taking the set
\[X=\mathbb{R}^2\]
with the formula
\[q((x_1,x_2),t,(y_1,y_2))=\]
\[((1-t)\,x_1+t\, y_1+k (y_2-x_2)^2 (1-t) t, (1-t)\, x_2+ t\, y_2).\]
\end{example}
\noindent In this example, when $y_2>y_1$, the operation $q((x_1,y_1),t,(x_2,y_2))$ simply gives the position at instant $t$ of a projectile in one-dimensional classical mechanics with constant acceleration $a_x=-2 k$, that is moving from position $x_1$ at time $y_1$ to position $x_2$ at time $y_2$.
\begin{example}\label{ex:counterexample}
 If a one-dimensional space $X$ would have been considered to describe the motion of the projectile instead of the Euclidean space-time of Example \ref{projectiles}, the ternary operation would have been:
\[q(x_1,t,x_2)=(1-t)\,x_1+t\, x_2+k (1-0)^2 (1-t) t.\]
This operation does not verify all the axioms of Definition \ref{mobi_space}. In particular, if the particle is not at rest, we will never get the idempotency $q(x,t,x)=x$. Axiom \ref{space_homo} is not verified either.

\end{example}

Most of the operations $q:X\times A\times X \to X$ that we can think of will not verify some of the axioms of Definition \ref{mobi_space}. Let us just point out that a simple example such as $q(x,t,y)=x+t^2(y-x)$, for $x,y$ in some subset of $\R^n$ and $t$ in a non trivial subset of $\R$, does not in general verify \ref{space_homo}.

There are also examples where an operation $q$ might verify the axioms of a mobi space but has no mobi space associated (simply because it may not be everywhere defined). That would be the case if Example \ref{ex11} would be generalized to Special Relativity \cite{relativistic_projectiles}. However, in Minkowski space-time, not every two points can be reached from one another if one point is not inside the {\it light cone} of the other.

\begin{example}\label{ex:fxdot}

Finally, let us observe that, in general, the solutions of $\ddot x=f(x,\dot x,t)$ (where $\dot x$ is the derivative of $x$ with respect to $t$), in $\mathbb{R}^n$, will not verify the axioms of a mobi space. Nevertheless, in $\mathbb{R}^{n+1}$, where the extra dimension is time, and if every two points can be reached from one another, then we can construct a mobi space as explain in Corollary \ref{corollary-xdot}. The case $f(x,\dot x,t)=-2k$ gives Example~\ref{projectiles}. To show another example, let us look at a critically damped harmonic oscillator corresponding to $f(x,\dot x,t)=-k^2 x-2k \dot x$. For any $k\in\R$, we obtain the following mobi space $(\mathbb{R}^2,q)$ over the unit interval 
with the formula
\[q((x_1,x_2),t,(y_1,y_2))=\]
\[((1-t)\,x_1\,e^{k t (x_2-y_2)}+t\, y_1\,e^{k (1-t) (y_2-x_2)}, (1-t)\, x_2+ t\, y_2).\]  

\end{example}

In the following section we will thoroughly analyse  examples occurring from spaces  with geodesic paths.

\section{Examples from Geodesics}\label{geo} 

In this section we analyse spaces which satisfy the equation for geodesics and observe that they all give rise to a mobi space. We have decided, for simplicity, to express the results within the scope of $\R^n$, but it is clear that the same principle will be valid for pseudo-Riemmanian manifolds with appropriate tensor metrics. This would, however, require a more sophisticated level of abstraction which is beyond our current purpose, namely to show the existence of examples of mobi spaces that are obtained from spaces with geodesics.

We also show that under suitable choices for appropriate sections, every identification space inherits the mobi space structure from its covering space. We illustrate this concept with the case of the cylinder (see Example \ref{ex5to9}(\ref{ex5}) and Example \ref{ex14}).

\subsection{Spaces with unique geodesics}
Let $X,V\subseteq \mathbb{R}^{n}$ be two open subsets of the $n$-dimensional space, with $V$ a vector space, and  $$g\colon{X\times V\to \mathbb{R}^{n}}$$ a map such that \begin{equation}
g(x,\lambda v)=\lambda^2g(x,v),\label{eq-1}
\end{equation} for all $x\in X$, $v\in V$ and $\lambda\in\mathbb{R}$. Moreover, the  initial value problem
\begin{equation}\label{initial value problem}
\left\lbrace\begin{matrix}
x'=v      &,& x(0)=x_0\in X\\
v'=g(x,v) &,& v(0)=v_0\in V
\end{matrix}
\right.
\end{equation}
 is supposed to have a unique solution, denoted by 
 \[x(t)=\pi(x_0,v_0,t),\]
 for every pair $(x_0,v_0)\in X\times V$ of initial conditions.  This means that, for fixed $x\in X$
 and $v\in V$, given any continuous map $f\colon{\mathbb{R}\to X}$, with continuous derivatives $f'\colon{\mathbb{R}\to V}$ and $f''\colon{\mathbb{R}\to \mathbb{R}^n}$, if \begin{equation}
 \label{eq-2}
 f(0)=x,\quad f'(0)=v,\quad f''(t)=g(f(t),f'(t))
\end{equation}  for every $t\in\mathbb{R}$, then 
\begin{equation}
f(t)=\pi(x,v,t).\label{eq-3}
\end{equation}
See Example \ref{ex13} for an illustration. Of course, the function $\pi$ has the following properties:
\begin{eqnarray}
\label{pi1} \pi(x,v,0)&=&x\\
\label{pi2} \pi'(x,v,0)&=&v\\
\label{pi3} \pi''(x,v,s)&=&g(\pi(x,v,s),\pi'(x,v,s)),
\end{eqnarray}
for every $s\in\mathbb{R}$, where $\pi'(x,v,s)$ stands for the derivative with respect to the variable~$s$.

\begin{lemma}\label{lemma 1}
Consider an initial value problem such as $(\ref{initial value problem})$. The following conditions hold for every $x\in X$, $v\in V$, and $s,t, u\in \mathbb{R}$:
\begin{eqnarray}
\label{lemma 1-1}\pi(x,0,t)&=&x\\
\label{lemma 1-2}\pi(x,v,s+u\,t)&=&\pi(\pi(x,v,s),u\,\pi'(x,v,s),t).
\end{eqnarray}
\end{lemma}
\begin{proof}
Considering the map $f(t)=x$ for the first case and the map $f(t)=\pi(x,v,s+ut)$ for the second case, $f$ satisfies the conditions (\ref{eq-2}) for each case. The desired equalities then follow from the uniqueness of solutions with respect to initial conditions, as expressed in condition~(\ref{eq-3}).
\end{proof}
\noindent Some useful particular cases of the previous lemma are:
\begin{eqnarray}
\label{lemma1-2}\pi(x,v,s\,t)&=&\pi(x,s\,v,t)\\
\pi(x,v,1-t)&=&\pi(\pi(x,v,1),-\pi'(x,v,1),t)\\
\label{lemma1-3}\pi(x,v,t)&=&\pi(\pi(x,v,s),(t-s)\pi'(x,v,s),1).
\end{eqnarray}
Property (\ref{lemma1-3}) follows from $\pi(x,v,t)=\pi(x,v,s+(t-s)\,1)$.

\begin{lemma}\label{lemma equivalent}
Consider an initial value problem such as $(\ref{initial value problem})$ and suppose the existence of a map $\beta\colon{X\times X\to V}$  such that 
\begin{equation}
\pi(x,\beta(x,y),1)=y,\label{lemmabeta1.1}\\
\end{equation}
for every $x,y\in X$. Then the following two conditions are equivalent:
\begin{enumerate}
\item[(a)] $\pi(x,v_1,1)=\pi(x,v_2,1)\Rightarrow v_1=v_2,\quad \forall x\in X, v_1,v_2\in V$
\item[(b)] the map $\beta$ with the property $(\ref{lemmabeta1.1})$ is unique.
\end{enumerate}
Moreover, we always have:
\[\beta(x,y_1)=\beta(x,y_2)\Rightarrow y_1=y_2.\]
\end{lemma}
\begin{proof}
Assuming condition $(a)$, if there exists another map with the same property as $\beta$, say $\hat{\beta}\colon{X\times X\to V}$, then, for every $x,y\in X$, we have\[\pi(x,\beta(x,y),1)=y=\pi(x,\hat{\beta}(x,y),1).\]Now, from $(a)$ it immediately follows that 
$\beta=\hat{\beta}$.
Conversely, let be given any $x\in X$ and $v_1,v_2\in V$ such that $$\pi(x,v_1,1)=\pi(x,v_2,1)$$ and let us denote by $y\in X$, both $\pi(x,v_1,1)$ and  $\pi(x,v_2,1)$. We now consider $\beta(x,y)\in V$ and observe that also $\pi(x,\beta(x,y),1)=y$ and hence, by the uniqueness of $\beta$, we obtain \[\beta(x,y)=v_1=v_2,\] which concludes the first part of the proof.
For the remaining part we simply observe that if $\beta(x,y_1)=\beta(x,y_2)$ then \[y_1=\pi(x,\beta(x,y_1),1)=\pi(x,\beta(x,y_2),1)=y_2.\]
\end{proof}
An interesting consequence of Lemma \ref{lemma equivalent}, when applied together with property (\ref{lemma1-2}), is that within an initial value problem such as $(\ref{initial value problem})$ where there exists a unique $\beta$ verifying (\ref{lemmabeta1.1}), we have, for $s\neq 0$, $x\in X$ and $v_1,v_2\in V$,
\begin{equation}\label{pi-cancel}
\pi(x,v_1,s)=\pi(x,v_2,s)\Rightarrow \pi(x,s\,v_1,1)=\pi(x,s\,v_2,1)\Rightarrow v_1=v_2.
\end{equation}

\begin{lemma}\label{lemma non unique beta}\label{lemma 2}
Consider an initial value problem such as $(\ref{initial value problem})$. If there exists a unique $\beta\colon{X\times X\to V}$  such that, for all $x, y \in X$,
\begin{equation}
\pi(x,\beta(x,y),1)=y\label{lemmabeta1}
\end{equation}
then we also have:
\begin{eqnarray}
\beta(x,x)&=&0\label{lemmabeta4}\\
s\,\beta(x,y)&=&\beta(x,\pi(x,\beta(x,y),s))\label{lemmabeta3}\\
\beta(\pi(x,v,s),\pi(x,v,t))&=&(t-s) \pi'(x,v,s)\label{lemmabeta5}.
\end{eqnarray}
for every $x\in X$, $v\in V$ and $t,s\in\mathbb{R}$.
\end{lemma}
\begin{proof}
From (\ref{lemma 1-1}), we have $\pi(x,0,1)=x$ and so $\beta(x,x)=0$ because $\beta(x,x)$ is the unique element in $V$ with the property $$\pi(x,\beta(x,x),1)=x.$$
Similarly we conclude, from (\ref{lemma1-2}) with $t=1$, that (\ref{lemmabeta3}) holds, 
indeed $\beta(x,\pi(x,\beta(x,y),s))$ is the unique element in $V$ such that 
\[\pi(x,\beta(x,\pi(x,\beta(x,y),s)),1)=\pi(x,\beta(x,y),s).\]
Finally, we observe that from (\ref{lemmabeta1}) we get:
\[\pi(\pi(x,v,s),\beta(\pi(x,v,s),\pi(x,v,t)),1)=\pi(x,v,t)\]
which, using (\ref{lemma1-3}), implies
\begin{eqnarray*}
&&\pi(\pi(x,v,s),(t-s)\pi'(x,v,s),1)\\
&=&\pi(\pi(x,v,s),\beta(\pi(x,v,s),\pi(x,v,t)),1).
\end{eqnarray*}
Therefore, the unicity of $\beta$ proves (\ref{lemmabeta5}).
\end{proof}
\noindent A particular case of the previous lemma is:
\begin{eqnarray}
\pi'(x,\beta(x,y),1)&=&-\beta(y,x)\label{lemmabeta2}.
\end{eqnarray}

\begin{theorem}\label{prop unique beta}
Consider an initial value problem such as $(\ref{initial value problem})$ in which there is a unique $\beta\colon{X\times X\to V}$  such that $(\ref{lemmabeta1})$ holds. Then, the structure $(X,q)$, with \[q(x,t,y)=\pi(x,\beta(x,y),t)\] is a mobi space over the mobi algebra $([0,1],p,0,\frac{1}{2},1)$, where \[p(s,u,t)=s+u\,(t-s).\]
\end{theorem}
\begin{proof}
\ref{space_0} is simply a consequence of (\ref{pi1}) and \ref{space_1} of (\ref{lemmabeta1}):
\begin{eqnarray*}
q(x,0,y)=\pi(x,\beta(x,y),0)&=&x,\\
q(x,1,y)=\pi(x,\beta(x,y),1)&=&y.
\end{eqnarray*}
To prove \ref{space_idem}, we use (\ref{lemmabeta4}) and (\ref{lemma 1-1}):
\[q(x,t,x)=\pi(x,\beta(x,x),t)=\pi(x,0,t)=x.\]
Lemma \ref{lemma equivalent} and property (\ref{pi-cancel}) imply \ref{space_cancel}:
\begin{eqnarray*}
q(x,\muu,y_1)=q(x,\muu,y_2)&\implies& \pi(x,\beta(x,y_1),\muu)=\pi(x,\beta(x,y_2),\muu)\\
                           &\implies& \beta(x,y_1)=\beta(x,y_2)\\
													 &\implies& y_1=y_2.
\end{eqnarray*}
Using (\ref{lemma 1-2}) and (\ref{lemmabeta5}), we obtain \ref{space_homo}:
\begin{eqnarray*}
&&q(q(x,s,y),u,q(x,t,y))\\
&=& \pi(\pi(x,\beta(x,y),s),\beta(\pi(x,\beta(x,y),s),\pi(x,\beta(x,y),t)),u)\\
&=& \pi(\pi(x,\beta(x,y),s),(t-s)\pi'(x,\beta(x,y),s),u)\\
&=& \pi(x,\beta(x,y),s+u\,(t-s))\\
&=& q(x,p(s,u,t),y).
\end{eqnarray*}

\end{proof}

Note that, when $s\neq 0$, from $(\ref{pi-cancel})$, we also deduce that $q(x,s,y_1)=q(x,s,y_2)$ implies  $y_1=y_2$.

Furthermore, in this case, the property (\ref{eq: conjecture}) is verified.

\begin{corollary}\label{corollary-geo} If $S\subseteq \mathbb{R}^n$ is a Riemann surface with a unique geodesic path between any two points, then $(S,q)$ is a mobi space with $q(x,t,y)$ being the position at an instant $t$ on the geodesic path between $x$ and $y$.
\end{corollary}
\begin{proof} If  $\Gamma_{ij}^{k}\colon{S\to \mathbb{R}}$  are the Christoffel symbols (see \cite{Ryder}, for instance) for the metric, then the function $g$ in the initial value problem (\ref{initial value problem}) is of the form \[g_k(x,y)=-\sum_{i,j}y_{i}\Gamma_{ij}^{k}(x)y_j\] and henceforth satisfies  (\ref{eq-1}). The existence of unique geodesic paths between any two points implies the uniqueness of $\beta$ in Theorem \ref{prop unique beta}.
\end{proof}

\begin{corollary}\label{corollary-xdot}
Consider the functions $f:\mathbb{R}^n\times \mathbb{R}^n\times \mathbb{R}\to \mathbb{R}^n$ and $x: \mathbb{R}\to\mathbb{R}^n$, and let $\dot x$ and $\ddot x$ be the first and second derivatives of $x$ with respect to the variable $t\in\mathbb{R}$. If the following problem
\begin{equation}\label{xdot-problem}
\left\lbrace\begin{matrix}
\ddot x=f(x,\dot x,t)\\
x(t_1)=x_1\\
x(t_2)=x_2
\end{matrix}
\right.
\end{equation}
has a unique solution for any $x_1, x_2\in\mathbb{R}^n$ and $t_1, t_2\in\mathbb{R}$,  $t_2\neq t_1$, expressed as $x(t)=F(x_1,t_1,x_2,t_2,t)$, then $(\mathbb{R}^{n+1},q)$ is a mobi space over the mobi algebra $([0,1],p,0,\frac{1}{2},1)$, with
$$p\left(t_1,s,t_2\right)=t_1+s(t_2-t_1),$$
$$q\left((x_1,t_1),s,(x_2,t_2)\right)=\left(F(x_1,t_1,x_2,t_2,p(t_1,s,t_2)),p(t_1,s,t_2)\right).$$
\end{corollary}
\begin{proof} Considering the variable $s$ such that $p(t_1,s,t_2)=t$, and denoting by $x'$ and $x''$ the first and second derivatives of $x$ with respect to s, we get $x'=t' \dot x$ and $x''=t'^2 \ddot x$. With $\chi=(x,t)\in\mathbb{R}^{n+1}$ and $\omega=(x',t')$, we then have:
\begin{equation*}
\left\lbrace\begin{matrix}
\chi'=\omega\\
\omega'=g(\chi,\omega)
\end{matrix}
\right.\ , \textrm{with } g((x,t),(x',t'))=(t'^2 f(x,\frac{x'}{t'},t),0).
\end{equation*}
It is obvious that $g(\chi,\lambda \omega)=\lambda^2 g(\chi,\omega)$. The unicity of a solution of $(\ref{xdot-problem})$ implies that there exists a initial value problem (at the initial value $s=0$), such as $(\ref{initial value problem})$, equivalent to $(\ref{xdot-problem})$. Therefore, Theorem \ref{prop unique beta} can be applied and proves this Corollary.
\end{proof}
The mobi spaces of Examples \ref{projectiles}  and \ref{ex:fxdot} are illustrations of the result of  Corollary \ref{corollary-xdot}.

\begin{example}\label{ex13}

To illustrate the result of Theorem \ref{prop unique beta}, consider the explicit example where $X=\R^2$, $V=\R^2$ and, with $k\in R$,
\begin{eqnarray*}
g:X\times V &\rightarrow & \mathbb{R}^2\\
((a_1,a_2),(w_1,w_2) )&\rightarrow & (-2 k w_2^2,0).
\end{eqnarray*}
The solution of the following initial value problem, involving the functions $x:\mathbb{R}\rightarrow X$ and $v:\mathbb{R}\rightarrow V$,
\begin{equation}
\left\lbrace\begin{matrix}
x'=v      &,& x(0)=(x_1,x_2)\in X\\
v'=g(x,v) &,& v(0)=(v_1,v_2)\in V
\end{matrix}
\right.
\end{equation}
is
$$ \pi((x_1,x_2),(v_1,v_2),t)=(x_1+v_1 t-k v_2^2 t^2,x_2+v_2 t)$$
It is easy to see that there is a unique $\beta:X\times X\rightarrow V$ such that
$$\pi((x_1,x_2),\beta((x_1,x_2),(y_1,y_2)),1)=(y_1,y_2)$$ given by:
$$\beta((x_1,x_2),(y_1,y_2))=(y_1-x_1+k (y_2-x_2)^2,y_2-x_2).$$
We then obtain the ternary operation $q:X\times [0,1]\times X\rightarrow X$ of the mobi space $(X,q)$ which is given by:
\begin{eqnarray*}
&&q((x_1,x_2),t,(y_1,y_2))=\pi((x_1,x_2),\beta((x_1,x_2),(y_1,y_2)),t)\\
&=&(x_1+(y_1-x_1) t+k (y_2-x_2)^2 (t-t^2), x_2+(y_2-x_2)t).
\end{eqnarray*}
This is, in fact, Example \ref{ex11}. For an explicit example of geodesics on a two dimensional surface, we refer to \cite{scripta2}.
\end{example}

Metric spaces in which there is a unique geodesic through any two points have been characterized in \cite{busemann-1943} (see also \cite{Kock}).

\subsection{The cylinder as an identification space and its geodesics}\label{subsec: identification spaces}
If the uniqueness of the geodesic paths in a Riemannian surface is not guaranteed, we may still choose appropriate values for the map $\beta$ as in  Proposition~\ref{prop unique beta} and obtain a mobi space. For example, although, on a cylinder,  there are infinitely many geodesics through two points that do not lie on the same circle, we can construct mobi spaces over the set of points of a cylinder as illustrated in Example \ref{ex5to9}(\ref{ex5}) of Section  \ref{sec_examples}. See, for example \cite{busemann-1955} for the study of spaces in which a geodesic through two points is not unique.

\begin{proposition}\label{prop:sphere}
Let $(X,q)$ be a mobi space over the mobi algebra $(A,p,0,\muu,1)$ and $h\colon{X\to S}$ a map onto a set $S$. If we can find two maps $s\colon{S\to X}$ and $\theta\colon{S\times S\to X}$, satisfying the following conditions for every $u,v,v'\in S$ $a_1,a_2,a_3\in A$
\begin{eqnarray*}
hs(u) &=& u\\
h\theta(u,v) &=& v\\
\theta(u,u)&=&s(u)\\
  hq(s(u),\muu,\theta(u,v))&=&hq(s(u),\muu,\theta(u,v'))\Rightarrow  v=v' \\
hq(s(u),p(a_1,a_2,a_3),\theta(u,v))&=& hq(shq(s(u),a_1,\theta(u,v)),a_2,\cdots\\
&& \theta(hq(s(u),a_1,\theta(u,v)),hq(s(u),\cdots\\
&& a_3,\theta(u,v))))
\end{eqnarray*}
then $(S,q_S)$ is a mobi space over $A$ with $$q_S(u,a,v)=hq(s(u),a,\theta(u,v)),$$ for every $u,v\in S$ and $a\in A$. 
\end{proposition}

\begin{proof}
It is clear that the axioms of a mobi space follow from the conditions that are assumed to be satisfied by the three maps.
\end{proof}

As an example we show how to obtain geodesic paths on a cylinder from geodesic paths on a plane.

\begin{example}\label{ex14}
 let $(X,q)$ be the euclidean plane with the usual geodesic paths over the mobi algebra of the unit interval. Denote by $S$ the set $[0,2\pi[\times \mathbb{R}$ and define three maps: $h$, $s$ and $\theta$. The map $h\colon{X\to S}$ is defined as $h(x,y)=(x\mod 2\pi, y)$, the map  $s\colon{S\to X}$ is the inclusion map and $\theta\colon{S\times S\to X}$ is defined as
\begin{equation*}
\theta((x_1,y_1),(x_2,y_2))=
\left\{\begin{array}{lcl}
(x_2+2\pi,y_2) & \text{if} &  x_2-x_1< -\pi\\
(x_2,y_2) & \text{if} & -\pi\leq x_2-x_1\leq \pi\\
(x_2-2\pi,y_2) & \text{if} &  \pi < x_2-x_1
\end{array}\right..
\end{equation*}
It can be checked that these maps satisfy the conditions of the previous proposition and hence they define a mobi space on the set $S$. 
\end{example}

The resulting mobi space structure in this case is not the same as the ones presented in Example \ref{ex5to9}(\ref{ex5}). Indeed, in this case, we obtain a parametrization for the cylinder which induces, via the mobi space, the shortest path between any two points. That was not the case in Example \ref{ex5to9}(\ref{ex5}).


\section{Comparison  with R-modules}\label{sec:modules}

Consider a unitary ring $(R,+,\cdot,0,1)$. It has been proven \cite{mobi} that if $R$ contains the inverse of $1+1$, then it is a mobi algebra and if a mobi algebra $(A,p,0,\muu,1)$ contains the inverse of $\muu$, in the sense of (\ref{def_product}), then it is a unitary ring. In this section, we will compare a module over a ring $R$ with a mobi space over a mobi algebra A. First, let us just recall that a module over a ring $R$ is a system $(M,+,e,\varphi)$, where $\varphi:R\to End(M)$ is a map from $R$ to the usual ring of endomorphisms, such that $(M,+,e)$ is an abelian group and $\varphi$ is a ring homomorphism.

The following Theorem shows how to construct a mobi space from a module over a ring containing the inverse of $2$.
\begin{theorem}\label{module2mobi}
Consider a module $(X,+,e,\varphi)$ over a unitary ring $(A,+,\cdot,0,1)$. If $A$ contains $(1+1)^{-1}=\muu$ then $(X,q)$ is an affine mobi space over the mobi algebra $(A,p,0,\muu,1)$, with
\begin{eqnarray}
p(a,b,c)&=&a+b c-b a\\
q(x,a,y)&=&\varphi_{1-a}(x)+\varphi_a(y).\label{q}
\end{eqnarray}
\end{theorem}
\begin{proof}
$(A,p,0,\muu,1)$ is a mobi algebra by Theorem 7.2 of \cite{mobi}. We show here that the axioms of Definition \ref{mobi_space}, as well as (\ref{affine}), are verified. The first three axioms are easily proved:
\begin{eqnarray*}
q(x,0,y)&=&\varphi_1(x)+\varphi_0(y)=x+ e=x\\
q(x,1,y)&=&\varphi_0(x)+\varphi_1(y)=e+ y=y\\
q(x,a,x)&=&\varphi_{1-a}(x)+\varphi_a(x)=\varphi_{1-a+a}(x)=\varphi_1(x)=x.
\end{eqnarray*}
Axiom \ref{space_cancel} is due to the fact that $\muu+\muu=1$ and consequently
\begin{eqnarray*}
\varphi_\muu(y_1)=\varphi_\muu(y_2)&\Rightarrow& \varphi_\muu(y_1)+\varphi_\muu(y_1)=\varphi_\muu(y_2)+\varphi_\muu(y_2)\\                                   &\Rightarrow& \varphi_1(y_1)=\varphi_1(y_2)\Rightarrow y_1=y_2.
\end{eqnarray*}
Next, we give a proof of Axiom \ref{space_homo}. It is relevant to notice that, besides other evident properties of the module $X$, the associativity of~$+$ plays an important part in the proof:
\begin{eqnarray*}
&&q(q(x,a,y),b,q(x,c,y))\\
&=&\varphi_{1-b}(\varphi_{1-a}(x)+\varphi_a(y))+\varphi_b(\varphi_{1-c}(x)+\varphi_c(y))\\
&=&\varphi_{1-b}(\varphi_{1-a}(x))+\varphi_{1-b}(\varphi_a(y))+\varphi_b(\varphi_{1-c}(x))+\varphi_b(\varphi_c(y))\\
&=&\varphi_{(1-b)(1-a)}(x)+\varphi_{b(1-c)}(x)+\varphi_{(1-b)a}(y)+\varphi_{bc}(y)\\
&=&\varphi_{1-a+ba-bc}(x)+\varphi_{a-ba+bc}(y)\\
&=&\varphi_{1-p(a,b,c)}(x)+\varphi_{p(a,b,c)}(y)\\
&=&q(x,p(a,b,c),y).
\end{eqnarray*}
It remains to prove (\ref{affine}):
\begin{eqnarray*}
&&q(q(x_1,a,y_1),\muu,q(x_2,a,y_2))\\
&=&\varphi_{\muu}(\varphi_{1-a}(x_1)+\varphi_a(y_1))+\varphi_\muu(\varphi_{1-a}(x_2)+\varphi_a(y_2))\\
&=&\varphi_{\muu}(\varphi_{1-a}(x_1)+\varphi_a(y_1)+\varphi_{(1-a)}(x_2)+\varphi_a(y_2))\\
&=&\varphi_{\muu}(\varphi_{1-a}(x_1+ x_2)+\varphi_a(y_1+ y_2))\\
&=&\varphi_{(1-a)\muu}(x_1+ x_2)+\varphi_{a\muu}(y_1+ y_2)\\
&=&\varphi_{(1-a)}(\varphi_\muu(x_1)+\varphi_\muu(x_2))+\varphi_a(\varphi_{\muu}(y_1)+\varphi_{\muu}(y_2))\\
&=&\varphi_{(1-a)}(q(x_1,\muu,x_2))+\varphi_{a}(q(y_1,\muu,y_2))\\
&=&q(q(x_1,\muu,x_2)),a,q(y_1,\muu,y_2)).
\end{eqnarray*}
\end{proof}

\begin{theorem}\label{mobi2module}
Consider a mobi space $(X,q)$, with a fixed chosen element $e\in X$, over a mobi algebra $(A,p,0,\muu,1)$. If $A$ contains $2$ such that $p(0,\muu,2)=1$ then $(X,+,e,\varphi)$ is a module over the unitary ring $(A,+,\cdot,0,1)$, with
\begin{eqnarray}
a+b&=&p(0,2,p(a,\muu,b))\\
a\cdot b&=&p(0,a,b)\\
\varphi_a(x)&=&q(e,a,x)\\
x+ y&=&q(e,2,q(x,\muu,y))=\varphi_2(q(x,\muu,y)).
\end{eqnarray}
\end{theorem}
\begin{proof}
$(A,+,\cdot,0,1)$ is a unitary ring by Theorem 7.1 of \cite{mobi}. We prove here that $(X,+,e,\varphi)$ is a module over $A$. First, we observe that, using in particular \ref{Y3} of Proposition \ref{properties_space}, we have:
\begin{eqnarray*}
q(e,\muu,x+y)&=& q(e,\muu,q(e,2,q(x,\muu,y)))\\
&=& q(e,\muu\cdot 2,q(x,\muu,y))\\
                            &=& q(e,1,q(x,\muu,y))\\
														&=& q(x,\muu,y).
\end{eqnarray*}
Then, the property (\ref{affine}) of an affine mobi space is essential to prove the associativity of the operation $+$ of the module:
\begin{eqnarray*}
q(e,\muu,q(e,\muu,(x+ y)+ z))&=&q(q(e,\muu,e),\muu,q(x+ y,\muu,z))\\
&=& q(q(e,\muu,x+ y),\muu,q(e,\muu,z))\\
&=& q(q(x,\muu,y),\muu,q(e,\muu,z))\\
&=& q(q(x,\muu,e),\muu,q(y,\muu,z))\\
&=& q(q(e,\muu,x),\muu,q(e,\muu,y+ z))\\
&=& q(q(e,\muu,e),\muu,q(x,\muu,y+ z))\\
&=& q(e,\muu,q(e,\muu,x+(y+ z)))\\
\end{eqnarray*}
Which, by \ref{space_cancel}, implies that $(x+ y)+ z=x+(y+ z)$. Commutativity of $+$ and the identity nature of $e$ are easily proved:
\begin{eqnarray*}
q(e,\muu,e+ x)&=&q(e,\muu,x)\Rightarrow e+ x=x\\
q(e,\muu,x+ y)&=&q(x,\muu,y)=q(y,\muu,x)\\
                    &=&q(e,\muu,y+ x)\Rightarrow x+ y=y+ x.
\end{eqnarray*}
Cancellation is achieved with $-x=q(e,p(1,2,0),x)$. Indeed:
\begin{eqnarray*}
q(e,\muu,q(e,p(1,2,0),x)+ x)&=&q(q(e,p(1,2,0),x),\muu,x)\\
                                  &=&q(q(e,p(1,2,0),x),\muu,q(e,1,x))\\
                                  &=&q(e,p(p(1,2,0),\muu,1),x)\\
                                  &=&q(e,p(1,p(2,\muu,0),0),x)\\
                                  &=&q(e,p(1,1,0),x)\\
                                  &=&q(e,0,x)=e=q(e,\muu,e)\\
\end{eqnarray*}
To prove that $\varphi_a(x+y)=\varphi_a(x)+\varphi_a(y)$, we will again need (\ref{affine}):
\begin{eqnarray*}
q(e,\muu,\varphi_a(x+y))&=& q(e,\muu,q(e,a,x+ y))\\    
														   &=& q(q(e,a,e),\muu,q(e,a,x+ y))\\   
														   &=& q(q(e,\muu,e),a,q(e,\muu,x+ y))\\   
														   &=& q(e,a,q(x,\muu,y))\\ 
														   &=& q(q(e,a,x),\muu,q(e,a,y))\\  
														   &=& q(\varphi_a(x),\muu,\varphi_a(y))\\ 
														   &=& q(e,\muu,\varphi_a(x)+\varphi_a(y)).
\end{eqnarray*}
To prove that $\varphi_{a+b}(x)=\varphi_a(x)+\varphi_b(x)$, let us first recall that, in a mobi algebra with 2 and $a+b=p(0,2,p(a,\muu,b))$, we have the following property:
$$p(0,\muu,a+b)=p(a,\muu,b).$$
We then have
\begin{eqnarray*}
q(e,\muu,\varphi_{a+b}(x))&=& q(e,\muu,q(e,a+b,x))\\    
													&=& q(q(e,0,x),\muu,q(e,a+b,x))\\    
													&=& q(e,p(0,\muu,a+b),x)\\    
													&=& q(e,p(a,\muu,b),x)\\    
													&=& q(q(e,a,x),\muu,q(e,b,x))\\    
													&=& q(\varphi_a(x),\muu,\varphi_b(x))\\
													&=& q(e,\muu,\varphi_a(x)+\varphi_b(x)).
\end{eqnarray*}
Last two properties are easily proved:
$$\varphi_{a\cdot b}(x)=q(e,a\cdot b,x)=q(e,a,q(e,b,x))=\varphi_a(\varphi_b(x))$$
$$\varphi_1(x)=q(e,1,x)=x.$$
\end{proof}

\begin{proposition}\label{module2module}
Consider a R-module $(X,+,e,\varphi)$ within the conditions of Theorem \ref{module2mobi} and the corresponding mobi space $(X,q)$. Then the R-module obtained from $(X,q)$ by Theorem \ref{mobi2module} is the same as $(X,+,e,\varphi)$.
\end{proposition}
\begin{proof}
From $(X,q)$, we define
$$x+'y=q(e,2,q(x,\muu,y))\ \textrm{and}\ \varphi'(x)=q(e,a,x)$$
and obtain the following equalities:
\begin{eqnarray*}
x+'y&=&e+\varphi_2(q(x,\muu,y))\\
    &=&\varphi_2(\varphi_\muu(x)+\varphi_\muu(y))\\
		&=&\varphi_{2\cdot\muu}(x)+\varphi_{2\cdot\muu}(y)\\
		&=&x+y
\end{eqnarray*}
$$\varphi'_a(x)=\varphi_{1-a}(e)+\varphi_a(x)=e+\varphi_a(x)=\varphi_a(x).$$
\end{proof}

\begin{proposition}\label{mobi2mobi}
Consider an affine mobi space $(X,q)$ within the conditions of Theorem \ref{mobi2module} and the corresponding module $(X,+,e,\varphi)$. Then the affine mobi space obtained from $(X,+,e,\varphi)$ by Theorem \ref{module2mobi} is the same as $(X,q)$.
\end{proposition}
\begin{proof}
From $(X,+,e,\varphi)$, we define
$$ q'(x,a,y)=\varphi_{1-a}(x)+\varphi_a(y)$$
and obtain the following equalities:
\begin{eqnarray*}
q'(x,a,y)&=& q(e,\overline{a},x)+q(e,a,y)\\
         &=& q(x,a,e)+q(e,a,y)\\
         &=& q(e,2,q(q(x,a,e),\muu,q(e,a,y))).
\end{eqnarray*}
Now, because we are considering that $(X,q)$ is affine, we get:
\begin{eqnarray*}
q'(x,a,y)&=& q(q(e,a,e),2,q(q(x,\muu,e),a,q(e,\muu,y)))\\
         &=& q(q(e,2,q(e,\muu,x)),a,q(e,2,q(e,\muu,y)))\\
         &=& q(q(e,2\cdot \muu,x),a,q(e,2\cdot\muu,y))\\
				 &=& q(x,a,y).
\end{eqnarray*}
\end{proof}

We have completely characterized affine mobi spaces in terms of modules over a unitary ring in which 2 is invertible. In a sequel to this paper we will investigate how to characterize mobi spaces in terms of  homomorphisms between mobi algebras. We will also dedicate our attention to the case of non affine mobi spaces.

\vspace{0.5cm}

We finish this section by taking a closer look to Example \ref{ex11} and Example \ref{ex5to9}(\ref{ex8}).

 Example \ref{ex11} is an example of an affine mobi space and can be extended to the case where the underlying mobi algebra is $(\mathbb{R},p,0,\frac{1}{2},1)$. Then by Theorem \ref{mobi2module} we get a module over $(\R,+,\cdot,0,1)$ given by $(\R^2,+,(0,0),\varphi)$ with
\begin{eqnarray*}
(x_1,x_2)+(y_1,y_2)&=&(x_1+y_1-2 k x_2 y_2,x_2+y_2)\\
\varphi_a(x_1,x_2)&=&(a x_1+k (1-a) a x_2^2,a x_2).
\end{eqnarray*}
By Theorem \ref{module2mobi}, we can construct a mobi space from this module and verify that it is the same as Example \ref{ex11}. Of course, in this example, the module is a vector field and there is a homomorphism, namely
$$f(x_1,x_2)=(x_1+k (x_2^2-x_2),x_2)$$
from $(\R^2,+,(0,0),\varphi)$ to the usual vector fied in $\R^2$.

Example \ref{ex5to9}(\ref{ex8}) is a mobi space that it is not affine. Indeed (\ref{affine}) is not always verified as, for instance, we have:
$$q(q((0,0),\frac{1}{3},(1,0)),\frac{1}{2},q((1,1),\frac{1}{3},(0,0)))=(\frac{1}{2},\frac{19}{189})$$
while
$$q(q((0,0),\frac{1}{2},(1,1)),\frac{1}{3},q((1,0),\frac{1}{2},(0,0)))=(\frac{1}{2},\frac{1}{12}).$$
We can also extend its underlying mobi algebra to $(\mathbb{R},p,0,\frac{1}{2},1)$. Theorem \ref{mobi2module}  cannot be applied but we can still find the corresponding $+$ operation and $\varphi_a$ functions. Choosing $e=(0,0)$, the results are the following:
\begin{eqnarray*}
&&(x_1,x_2)+(y_1,y_2)\\
&=&\left(x_1+y_1,\dfrac{(x_1^2+4x_1 y_1+7 y_1^2) x_2+(7x_1^2+4x_1y_1+y_1^2) y_2}{x_1^2+x_1y_1+y_1^2}\right)
\end{eqnarray*}
for $( x_1,y_1)\neq(0,0)$ and 
$$(0,x_2)+(0,y_2)=(0,4x_2+4y_2);$$
$$\varphi_a(x_1,x_2)=(a x_1,a^3 x_2).$$
The operation $+$ is commutative, admits $e$ as a identity and the symmetric element $(-x_1,-x_2)$ for all $(x_1,x_2)\in \R^2$ but it is not associative. On the other side, we have:
\begin{eqnarray*}
\varphi_a(x+y)&=&\varphi_{a}(x)+\varphi_{a}(y)\\
\varphi_{a+b}(x)&=&\varphi_{a}(x)+\varphi_{b}(x)\\
\varphi_{a b}(x)&=&\varphi_a(\varphi_b(x))\\
\varphi_{1}(x)&=&x\\
\varphi_{0}(x)&=&e\\
\end{eqnarray*}
Beginning with the structure $(X,+,e,\varphi)$, we can construct a ternary operation $q$ on $X$ given by (\ref{q}) but then $(X,q)$ is not a mobi space because it doesn't verify axiom \ref{space_homo}. This example shows that a mobi space is richer than a structure of the type $(X,+,\varphi)$.

\end{document}